\documentclass[11pt,a4paper]{article}
\usepackage[utf8]{inputenc}
\usepackage{amsmath,amssymb,amsthm}
\usepackage{graphicx}
 \scrollmode

\title{Approximation of discontinuous signals by sampling Kantorovich series}
         
\author{{\bf Danilo Costarelli}, \hskip0.6cm {\bf Anna Maria Minotti}, \\ and \hskip0.4cm {\bf Gianluca Vinti} \thanks{corresponding author}    \\    \\
   Department of Mathematics and Computer Science \\
            University of Perugia\\
        1, Via Vanvitelli, 06123 Perugia, Italy    \\  \\
  {\small {\tt danilo.costarelli@unipg.it} \hskip0.3cm {\tt annamaria.minotti@unipg.it}}\\
{\small and \hskip0.4cm {\tt gianluca.vinti@unipg.it}} }

\date{}

\newcommand{\mau}{\geq}
\newcommand{\miu}{\leq}
\newcommand{\ep}{\varepsilon}
\newcommand{\N}{\mathbb{N}}
\newcommand{\R}{\mathbb{R}}
\newcommand{\Z}{\mathbb{Z}}
\newcommand{\disp}{\displaystyle}
\newcommand{\be}{\begin{equation}}
\newcommand{\ee}{\end{equation}}

\newcommand{\lef}{\left\lfloor}
\newcommand{\rig}{\right\rfloor}

\newtheorem{definition}{Definition}[section]
\newtheorem{remark}[definition]{Remark}
\newtheorem{theorem}[definition]{Theorem}
\newtheorem{lemma}[definition]{Lemma}

\begin{document}

\maketitle  

\begin{abstract}
In this paper, the behavior of the sampling Kantorovich operators has been studied, when discontinuous signals are considered in the above sampling series. Moreover, the rate of approximation for the family of the above operators is estimated, when uniformly continuous and bounded signals are considered. Further, also the problem of the linear prediction by sampling values from the past is analyzed. At the end, the role of duration-limited kernels in the previous approximation processes has been treated, and several examples have been provided.

\end{abstract}

\section{Introduction}

The present study is motivated by various reasons. First of all, we were inspired from the paper of Butzer, Ries and Stens \cite{BURIST1}, where the behavior of the generalized sampling operators, of the form:
$$
\mbox{(I)} \hskip1.8cm (G_w f)(t)\ :=\ \sum_{k \in \Z} f\left( \frac{k}{w} \right) \chi\left(wt-k\right), \hskip1cm t \in \R, \hskip0.8cm
$$ 
were analyzed at any point $t \in \R$, where the bounded signal $f$ defined on $\R$, presents a jump discontinuity. The function $\chi: \R \to \R$ is a suitable kernel, satisfying certain assumptions. The generalized sampling operators $G_w f$ have been introduced in the 1980s (see e.g. \cite{BAVI0,BAMA0,BURIST1,BUST2,RIST}) as an approximate version of the classical Whittaker-Kotelnikov-Shannon  sampling theorem which uses an $L^1$-kernel, see e.g. \cite{SHAN,BU1,HIG1,HIG2,HIST,ANVI,BABUSTVI0,BAKAVI,ANVI1,ANVI2,BABUSTVI2}.

   The sampling Kantorovich operators $S_w f$, first introduced in \cite{BABUSTVI} and studied in this paper (see Section \ref{sec2}), arise as a further development of the generalized sampling operators, where in place of the sample values $f\left( k/w \right)$ we have mean values of the signal $f$, of the form $w \int_{k/w}^{(k+1)/w}f(u)\, du$, $k \in \Z$, and $w>0$; it turns out that these operators reduce "time-jitter" errors, what is very useful in signal processing. 

  In view of the above connection between the operators $S_w f$ and $G_w f$, it is quite natural to ask what is the behavior of the sampling Kantorovich series when signals with discontinuities are considered. Moreover, in last years it has been proved that the sampling Kantorovich operators are very suitable for image reconstruction and enhancement (see e.g. \cite{COVI1,COVI2,CLCOMIVI1,CLCOMIVI2}); images are typical examples of discontinuous signals. This is an additional reason for which it would be interesting to study the behavior of the sampling Kantorovich operators at jump discontinuities. The latter problem will be the main topic investigated in the present paper. 

   First of all, we establish a representation formula (see Lemma \ref{lemma1}) for the sampling Kantorovich series evaluated at a certain fixed time $t$ where the signal $f$ has a jump discontinuity, exploiting an auxiliary function appropriately defined. Then, by using the above formula, we become able to obtain some necessary and sufficient conditions for the convergence of the family $(S_w f)_{w>0}$ to a suitable finite linear combination of $f(t+0)$ and $f(t-0)$, where $f(t+0)$ and $f(t-0)$ are respectively the right and left limit of $f$.

   The sampling Kantorovich operators, other than applications to image processing, have been largely studied also from the theoretical point of view. For instance, a nonlinear version of $S_w f$ has been studied in \cite{VIZA1,COVI2,VEVI1} both in continuous and Orlicz functions spaces, while a more general version of these operators have been considered in \cite{VIZA2,VIZA3}. We recall that, Orlicz spaces are very general spaces, including as a special case the $L^p$-spaces (see e.g. \cite{MUORL,MU1,BAMUVI}).  The rate of approximation for $S_w f$ has been investigated in \cite{BAMA,COVI3,COVI4,COVI5} both in univariate and multivariate settings. 

   In particular, Bardaro and Mantellini (\cite{BAMA}) proved an estimate for the order of approximation involving the modulus of continuity of the function being approximated, requiring that the discrete absolute moment of order $\beta \mau 1$ of the kernel used to construct the operators is finite (i.e. $m_{\beta}(\chi)<+\infty$, with $\beta \mau 1$). Since in general examples of kernels for which the discrete moment $m_{\beta}(\chi)=+\infty$, for $\beta \mau 1$, can be given, in the present paper we achieve an estimate applicable to sampling Kantorovich operators based upon those kernels. 

  Finally, we study the problem of the linear prediction of signals from sample values taken only from the past. This problem is important in real-word case studies, when, in order to reconstruct a given signal at time $t$, one knows the sample values only before the present time $t$. 

  In all the above mentioned problems, a crucial role is played by  the kernels used to construct the sampling Kantorovich operators. For this reason, a detailed discussion concerning the kernels is given at the end of the paper together with several examples. 


\section{Approximation of discontinuous signals}  \label{sec2}

First of all, we introduce the family of discrete operators studied in this paper. 

   In what follows, a function $\chi: \R \to \R$ will be called a {\em kernel} if it satisfies the following conditions:
\begin{itemize}

\item[$(\chi 1)$] $\chi \in L^1(\R)$ and it is bounded in $[-1, 1]$;

\item[$(\chi 2)$] for every $u \in \R$,
\be
\sum_{k \in \Z}\chi(u-k) = 1;
\ee

\item[$(\chi 3)$] for some $\beta>0$, the discrete absolute moment of order $\beta$ are finite, i.e.,
$$
m_{\beta}(\chi)\ :=\ \sup_{u \in \R} \sum_{k \in \Z}|\chi(u-k)|\, |u-k|^{\beta}\ <\ +\infty.
$$
\end{itemize}
For any kernel $\chi$, the sampling Kantorovich operators can be defined as follows:
$$
(S_w f)(t)\ :=\ \sum_{k \in \Z} \chi(wt-k) \left[ w \int_{k/w}^{(k+1)/w} f(u)\, du \right], \hskip0.8cm (t \in \R)
$$
where $f: \R \to \R$ is a locally integrable function such that the above series is convergent for each $t \in \R$.
\begin{remark} \rm \label{remark2.3}
Note that, since a kernel function satisfies conditions $(\chi 1)$ and $(\chi 3)$, it is possible to prove (see e.g. \cite{BABUSTVI,COVI1}) that the discrete absolute moment of order zero is finite, i.e., $m_0(\chi) < +\infty$.
\end{remark}
\begin{remark} \rm
We point out that, the boundedness on $[-1,1]$ required in assumption $(\chi 1)$ is a merely, non-restrictive, technical condition. From the practical point view, the above assumption is not restrictive since the main examples of kernels (that will we show in Section \ref{sec4}) are bounded on the whole $\R$.   
\end{remark}

   Clearly, it is easy to see that the operators $S_w f$ are well-defined for $f \in L^{\infty}(\R)$. Indeed, 
$$
|(S_w f)(t)|\ \miu\ \|f\|_{\infty}\, m_0(\chi)\ < +\infty, \hskip0.8cm t \in \R,
$$
where $\| \cdot \|_{\infty}$ denotes the usual sup-norm.
The pointwise and uniform convergence for the family of sampling Kantorovich operators have been proved in \cite{BABUSTVI} for continuous signals of one variable. More precisely, we have the following.
\begin{theorem}[\cite{BABUSTVI}] \label{th1}
Let $f:\R \to \R$ be a bounded function. Then,
$$
\lim_{w \to +\infty}(S_w f)(t) = f(t),
$$
at any point $t \in \R$ of continuity for $f$. Moreover, if $f$ is uniformly continuous, it turns out:
$$
\lim_{w \to +\infty} \| S_w f - f \|_{\infty}\ =\ 0.
$$
\end{theorem}
\noindent A multivariate version of the above theorem has been proved in \cite{COVI1}.

Note that, assumption $(\chi 1)$ on kernels functions $\chi$ is quite standard, condition $(\chi 3)$ can be easily deduced if $\chi(u) = {\cal O}(u^{-\beta - 1 - \ep})$, as $|u| \to + \infty$, for some $\ep >0$, while condition $(\chi 2)$ is in general difficult to check. For this reason, the following theorem can be useful.
\begin{theorem}[\cite{BURIST1}] \label{th_eq}
Let $\chi$ be a continuous kernel function. Then, the following two assertion are equivalent:
\begin{itemize}
\item[${\rm (I)}$] For every $u \in [0,1)$, \hskip0.5cm 
$$
\sum_{k \in \Z} \chi(u - k)\ =\ 1;
$$

\item[${\rm (II)}$] 
$$
\widehat{\chi}(2 \pi k)\ =\ \left\{
\begin{array}{l}
1,\ \hskip0.5cm k=0,\\
0,\ \hskip0.5cm k \in \Z\setminus\left\{ 0 \right\},
\end{array}
\right.
$$
where $\widehat{\chi}(v):= \int_{\R}\chi(u)\, e^{-iuv}\, du$, $v \in \R$, denotes the Fourier transform of $\chi$.
\end{itemize}
\end{theorem}
The proof of Theorem \ref{th_eq} follows as a consequence of the Poisson's summation formula, see e.g. \cite{BUNE,BURIST1,BUSPST}. Furthermore, it is easy to observe that (I) of Theorem \ref{th_eq} is equivalent to $(\chi 2)$ since the function $\sum_{k \in \Z}\chi(u-k)$ is 1-periodic. 

   Now, in order to investigate the behavior of the sampling Kantorovich series at points where a signal $f$ is discontinuous, we introduce some notations. 

   Let $f: \R \to \R$ and $t \in \R$ be fixed; we will denote by
$$
f(t+0)\ :=\ \lim_{\ep \to 0^+} f(t+\ep),\ \hskip0.2cm \mbox{and}\ \hskip0.2cm f(t-0)\ :=\ \lim_{\ep \to 0^+} f(t-\ep).
$$
Moreover, for the kernel $\chi$ we define the functions:
$$
\Psi^+_{\chi}(x)\, :=\, \sum_{k\, <\, x}\chi(x-k), \hskip0.7cm \Psi^-_{\chi}(x)\, :=\, \sum_{k\, >\, x}\chi(x-k), \hskip0.7cm (x \in \R).
$$
Under the previous assumptions, it is easy to see that $\Psi^+_{\chi}(x)$ and $\Psi^-_{\chi}(x)$ are periodic functions with period equal to one. Indeed:
$$
\Psi^+_{\chi}(x+1)\ =\ \sum_{k\, <\, x+1}\chi(x+1-k)\ =\ \sum_{k-1\, <\, x}\chi(x-(k-1))
$$
$$
=\ \sum_{\widetilde{k}\, <\, x}\chi(x-\widetilde{k})\ =\ \Psi^+_{\chi}(x), \hskip0.8cm (x \in \R),
$$
and the same computations can be made in case of $\Psi^-_{\chi}$.

  Moreover, assuming $\chi$ continuous, it is easy to prove that $\Psi^-_{\chi}$ is continuous from the right at the integers, and $\Psi^+_{\chi}$ is continuous from the left again at the integers, see e.g., \cite{BURIST1}.  

  Now, we first prove the following useful representation lemma for the operators $S_w f$.
\begin{lemma} \label{lemma1}
Let $f: \R \to \R$ be a bounded function, and $t \in \R$ be fixed. We define:
\be \label{g_t}
g_t(x)\ :=\ \left\{
\begin{array}{l}
f(x)-f(t-0), \hskip0.5cm x<t,\\
\\
0, \hskip2.9cm x=t, \\
\\
f(x)-f(t+0), \hskip0.5cm x>t.
\end{array}
\right.
\ee
There holds:
$$
(S_w f)(t)\ =\  (S_w g_t)(t)\,  +\, \left[ f(t+0) - f(t-0)  \right] \cdot \left[\Psi^-_{\chi}(w t) + \chi\left(w t - \lef wt\rig\right)  \right]
$$
$$
\hskip1cm +\, f(t-0)\, - \chi\left(w t - \lef wt \rig\right) \left( w t - \lef wt \rig \right) \left[ f(t+0) - f(t-0)  \right]
$$
if $wt \notin \Z$, $w>0$, where the symbol $\lef \cdot \rig$ denotes the integer part of a given number, while
$$
(S_w f)(t)\, =\, (S_w g_t)(t)\, +\, \left[ f(t+0) - f(t-0)  \right] \cdot \left[\Psi^-_{\chi}(w t) + \chi\left(0\right)  \right]\, +\, f(t-0),
$$
if $wt \in \Z$, $w>0$.
\end{lemma}
\begin{proof}
First we consider the case $w t \notin \Z$, with $w>0$. Thus we can write:
$$
\hskip-1.2cm (S_w g_t)(t)\ =\ \sum_{k < \lef wt \rig} \chi(wt-k)\left[ w \int_{k/w}^{(k+1)/w} \left[ f(u)-f(t-0)  \right]\, du \right]\ 
$$
$$
\hskip0.5cm +\ \sum_{k > \lef wt \rig} \chi(wt-k)\left[ w \int_{k/w}^{(k+1)/w} \left[ f(u)-f(t+0)  \right]\, du \right]\ 
$$
$$
\hskip-0.7cm +\ \chi(wt-\lef wt \rig)\, w \left[ \int_{\lef wt \rig / w}^t \left[ f(u)-f(t-0)  \right]\, du \right.
$$
$$
\hskip-2.3cm +\ \left.   \int_t^{(\lef wt \rig+1)/w} \left[ f(u)-f(t+0)  \right]\, du \right]
$$
$$
=\ \sum_{k < \lef wt \rig} \chi(wt-k)\left[ w \int_{k/w}^{(k+1)/w} f(u)\, du \right] - f(t-0)\sum_{k < \lef wt \rig} \chi(wt-k)
$$
$$
+\ \sum_{k > \lef wt \rig} \chi(wt-k)\left[ w \int_{k/w}^{(k+1)/w} f(u)\, du \right] - f(t+0)\sum_{k > \lef wt \rig} \chi(wt-k)
$$
$$
+\ \chi(wt-\lef wt \rig)\, w \left[\int_{\lef wt \rig / w}^{(\lef wt \rig+1)/w} f(u)\, du \right] 
$$
$$
-\ \chi(wt-\lef wt \rig)\, w \left[  f(t-0) \left( t-\frac{\lef wt \rig}{w}  \right) +\ f(t+0) \left( \frac{\lef wt \rig + 1}{w} - t \right)\right]
$$
$$
=\ (S_w f)(t)\, - f(t-0)\sum_{k < \lef wt \rig} \chi(wt-k) - f(t+0)\sum_{k > \lef wt \rig} \chi(wt-k)
$$
$$
-\ \chi(wt-\lef wt \rig) \left[  f(t-0) \left( wt- \lef wt \rig \right) +\ f(t+0) \left( \lef wt \rig + 1 - wt \right) \right].
$$
Rearranging all the above terms, and by adding then subtracting the term $f(t-0)\sum_{k \mau \lef wt \rig} \chi(wt-k)$ we obtain:
$$
(S_w f)(t)\ =\ (S_w g_t)(t) + f(t-0) \sum_{k \in \Z}\chi(wt-k) - f(t-0)\sum_{k > \lef wt \rig}\chi(wt-k)
$$
$$
-\ f(t-0)\, \chi(wt-\lef wt \rig) + f(t+0)\sum_{k > \lef wt \rig}\chi(wt-k) + f(t+0)\, \chi(wt-\lef wt \rig)
$$
$$
+\ \chi(wt-\lef wt \rig) \left[  f(t-0) \left( wt- \lef wt \rig \right) +\ f(t+0) \left( \lef wt \rig - wt \right) \right].
$$
Now, since $wt$ is not an integer, it is easy to note that:
$$
\Psi^-_{\chi}(wt)\ =\ \sum_{k> \lef wt \rig}\chi(wt - k),
$$
moreover, by using condition $(\chi 2)$ we finally have:
$$
(S_w f)(t)\ =\ (S_w g_t)(t) + f(t-0) + \Psi^-_{\chi}(wt) \left[  f(t+0) - f(t-0) \right]
$$
$$
+\ \chi(wt-\lef wt \rig)\, \left[  f(t+0) - f(t-0) \right] 
$$
$$
-\ \chi(wt-\lef wt \rig)\, \left( wt-\lef wt \rig  \right) \left[  f(t+0) - f(t-0) \right]. 
$$
While, if $wt \in \Z$, $w>0$, we can repeat the above computations, splitting the operator $(S_w g_t)(t)$ as follows:
$$
\hskip-1.2cm (S_w g_t)(t)\ =\ \sum_{k < wt} \chi(wt-k)\left[ w \int_{k/w}^{(k+1)/w} \left[ f(u)-f(t-0)  \right]\, du \right]\ 
$$
$$
\hskip0.5cm +\ \sum_{k \mau wt } \chi(wt-k)\left[ w \int_{k/w}^{(k+1)/w} \left[ f(u)-f(t+0)  \right]\, du \right].
$$
Hence, rearranging all terms, by adding then subtracting again the term $f(t-0)\sum_{k \mau  wt } \chi(wt-k)$, and by using condition $(\chi 2)$ we can obtain: 
$$
(S_w f)(t)\ =\ (S_w g_t)(t) + f(t-0) + \left[ f(t+0) - f(t-0)  \right]\sum_{k \mau wt}\chi(wt-k)
$$
$$
=\ (S_w g_t)(t)\, +\, f(t-0)\, +\, \left[ f(t+0) - f(t-0)  \right]\, \left[  \Psi^-_{\chi}(wt)\,  +\,  \chi(0) \right].
$$
\end{proof}
\begin{remark} \rm
The representation formula obtained in Lemma \ref{lemma1} it holds for every bounded signal $f$, and for signals having removable discontinuities, the two cases $w t \in \Z$ and $w t \notin \Z$ coincide. 
\end{remark}
We are now ready to study the behavior of the sampling Kantorovich series at jump discontinuity.
\begin{theorem} \label{th2}
Let $f: \R \to \R$ be a bounded signal with a (non removable) jump discontinuity at $t \in \R$, and let $\alpha \in \R$. Then, the following assertions are equivalent:
\begin{itemize}
\item[(i)] $\displaystyle \lim_{\stackrel{\displaystyle w \to +\infty}{\displaystyle wt \in \Z}} (S_wf)(t)\ =\ [\, \alpha + \chi(0)\, ]\, f(t+0) + [\, 1-\alpha - \chi(0)\, ]\, f(t-0)$;

\item[(ii)] $ \Psi^-_{\chi}(0)\ =\ \alpha$; 

\item[(iii)] $ \Psi^+_{\chi}(0)\ =\ 1-\alpha-\chi(0)$;
\end{itemize}
\end{theorem}
\begin{proof}
$(i)\Rightarrow(ii)$ From Lemma \ref{lemma1} we have:
\be \label{condizione}
(S_w f)(t) = (S_w g_t)(t) + \left[ f(t+0) - f(t-0)  \right] \cdot [\Psi^-_{\chi}(w t) + \chi(0)] + f(t-0),
\ee
for any $w>0$, such that $wt \in \Z$. Now, by definition, it turns out that $g_t$ is bounded, continuous in $t$, and $g_t(t)=0$, then from the assumptions and by Theorem \ref{th1} we have: 
$$
\hskip-3cm [\, \alpha\, +\, \chi(0)\, ]\, f(t+0)\, +\, [\,1 - \alpha-\chi(0)\,]\, f(t-0)\ =
$$
$$
=\ f(t-0)\, +\, \left[ f(t+0) - f(t-0)  \right]\, \left[\, \chi(0)\, +\, \lim_{\stackrel{\displaystyle w \to +\infty}{\displaystyle wt \in \Z}}  \Psi^-_{\chi}(w t)  \right].
$$ 
Since $f$ has a jump discontinuity in $t$, we have $f(t+0)-f(t-0) \neq 0$, then:
$$
\lim_{\stackrel{\displaystyle w \to +\infty}{\displaystyle wt \in \Z}}  \Psi^-_{\chi}(w t)\ =\ \alpha.
$$
Now, recalling that the function $\Psi^-_{\chi}$ is periodic with period equal to one, it turns out that:
$$
\lim_{\stackrel{\displaystyle w \to +\infty}{\displaystyle wt \in \Z}}  \Psi^-_{\chi}(w t)\ =\ \Psi^-_{\chi}(0)\ =\ \alpha.
$$
$(ii)\Rightarrow(i)$ Since $\Psi^-_{\chi}$ is 1-periodic, we have $\Psi^-_{\chi}(w t)\ =\ \Psi^-_{\chi}(0)\ =\ \alpha$, for every $w>0$ such that $wt \in \Z$. Then, by applying Theorem \ref{th1} again, and equality (\ref{condizione}) we easily obtain:
$$
\displaystyle \lim_{\stackrel{\displaystyle w \to +\infty}{\displaystyle wt \in \Z}} (S_wf)(t)\ =\ [\, \alpha\, +\, \chi(0)\, ]\, f(t+0)\, +\, [1-\alpha-\chi(0)\, ]\, f(t-0).
$$
Finally, by the assumption $(\chi 2)$ on the kernels, we can write:
$$
1\ =\ \sum_{k \in \Z} \chi(-k)\ =\ \Psi^-_{\chi}(0)\ +\ \chi(0)\ +\ \Psi^+_{\chi}(0),
$$
then immediately follows that (ii) is equivalent to (iii). This completes the proof.
\end{proof}
\begin{remark} \rm
We point out that in the case $\chi(0)=0$, (i) of Theorem \ref{th2} becomes:
$$
\lim_{\stackrel{\displaystyle w \to +\infty}{\displaystyle wt \in \Z}} (S_wf)(t)\ =\ \alpha\, f(t+0)\, +\, (1-\alpha )\, f(t-0),
$$
i.e., the sampling Kantorovich series converge to a value which does not depend by the behavior of the kernel at zero.
\end{remark}

   In what follows, when referring to the assertion (i) of Theorem \ref{th2}, we speak of convergence of the sampling Kantorovich series at jump discontinuities. 

  Now, in order to study the above problem when $w>0$ is such that $wt \notin \Z$, an additional assumption on $\chi$ must be assumed. In Section \ref{sec4}, we will show a class of kernel functions satisfying this assumption, and moreover, in Theorem \ref{th-zeros} we prove that the above additional assumption on $\chi$ cannot be dropped. We have the following.  
\begin{theorem} \label{th-bis}
Let $f: \R \to \R$ be a bounded signal with a (non removable) jump discontinuity at $t \in \R\setminus\left\{ 0 \right\}$, and let $\alpha \in \R$. Suppose in addition that the kernel function $\chi$ satisfies the following condition:
\be \label{strong}
\chi(u) = 0,\ \hskip0.8cm for\ every\ \hskip0.8cm x \in (0,1).
\ee
Then, the following assertions are equivalent:
\begin{itemize}
\item[(i)] $\displaystyle \lim_{\stackrel{\displaystyle w \to +\infty}{\displaystyle wt \notin \Z}} (S_wf)(t)\ =\ \alpha\, f(t+0)\, +\, (1-\alpha)\, f(t-0)$;

\item[(ii)] $ \Psi^-_{\chi}(x)\ =\ \alpha$, \hskip0.5cm for every  \hskip0.5cm $x \in (0, 1)$; 

\item[(iii)] $ \Psi^+_{\chi}(x)\ =\ 1-\alpha$ \hskip0.5cm for every  \hskip0.5cm $x \in (0, 1)$.
\end{itemize}
\end{theorem}
\begin{proof}
First of all, we can observe that under assumption (\ref{strong}) on $\chi$, by Lemma \ref{lemma1} we have:
\be \label{fuffa}
(S_w f)(t)\, =\, (S_w g_t)(t)\, +\, \left[ f(t+0) - f(t-0)  \right] \cdot \Psi^-_{\chi}(w t) \, +\, f(t-0),
\ee
for every $w>0$ such that $wt \notin \Z$, where $g_t$ is the function defined in (\ref{g_t}). Now, by using (\ref{fuffa}) it is immediate to prove that (ii) implies (i) by letting $w \to +\infty$ with the restriction $wt \notin \Z$, using Theorem \ref{th1}, and since, as in the proof of Theorem \ref{th2}, $\Psi^-_{\chi}$ is 1-periodic. 

\noindent Conversely, using (\ref{fuffa}) again, and since $f(t+0) \neq f(t-0),$ we can obtain:
$$
\lim_{\stackrel{\displaystyle w \to +\infty}{\displaystyle wt \notin \Z}}\Psi^-_{\chi}(wt)\ =\ \alpha,
$$
or, equivalently for every $x \in (0,1)$ and $n \in \N$
$$
\lim_{n \to +\infty}\Psi^-_{\chi}(x+n)\ =\ \Psi^-_{\chi}(x) =\ \alpha,
$$
since $\Psi^-_{\chi}$ is 1-periodic. Therefore (i) and (ii) are equivalent. 

\noindent The equivalence between (ii) and (iii) can be established easily since $\Psi^+_{\chi}$ and $\Psi^-_{\chi}$ are 1-periodic and by using condition $(\chi 2)$, we have:
$$
1\ =\ \sum_{k \in \Z} \chi(wt - k)\ =\ \Psi^+_{\chi}(wt)\ +\ \Psi^-_{\chi}(wt)\ =\ \Psi^+_{\chi}(x)\ +\ \Psi^-_{\chi}(x),
$$
for every $x$ in $(0,1)$.
\end{proof}
The results showed in Theorem \ref{th-bis}, have been proved by using the condition in (\ref{strong}), which could seem to be a quite strong assumption on the kernel $\chi$.

    Actually, it is possible to see that, even if condition (ii) of Theorem \ref{th-bis} is fulfilled, but $\chi$ does not satisfy (\ref{strong}), the sampling Kantorovich series cannot converge at jump discontinuities.
\begin{theorem} \label{th-zeros}
Let $\chi$ be a kernel which is not identically null on $(0,1)$. Suppose in addition that $\chi$ satisfies condition (ii) of Theorem \ref{th-bis} with $\alpha \in \R$.
Moreover, let $f: \R \to \R$ be a bounded signal with a (non removable) jump discontinuity at $t \in \R\setminus\left\{ 0 \right\}$. 

   Then, the family of the sampling Kantorovich operators based upon $\chi$, cannot converge (pointwise) at $t$.
\end{theorem}
\begin{proof}
Suppose by contradiction that:
$$
\displaystyle \lim_{\stackrel{\displaystyle w \to +\infty}{\displaystyle wt \notin \Z}} (S_wf)(t)\ =\ \ell,
$$
for some $\ell \in \R$. Then by the uniqueness of the limit, and by using Lemma \ref{lemma1} and Theorem \ref{th1} we have:
$$
\ell\ =\ \left[ f(t+0) - f(t-0)  \right] \cdot  \left\{ \lim_{\stackrel{\displaystyle w \to +\infty}{\displaystyle wt \notin \Z}}\left[\Psi^-_{\chi}(w t) + \chi\left(w t - \lef wt\rig\right)  \right]  \right\}
$$
$$
+\, f(t-0)\, - \left[ f(t+0) - f(t-0)  \right] \cdot \left\{ \lim_{\stackrel{\displaystyle w \to +\infty}{\displaystyle wt \notin \Z}} \chi\left(w t - \lef wt \rig\right) \cdot \left( w t - \lef wt \rig \right) \right\}.
$$
Now, by assumption (ii) of Theorem \ref{th-bis}, and noting that, for every $w>0$ with $wt \notin \Z$, we have $w t - \lef wt \rig = x \in (0,1)$, we can write what follows:
$$
\hskip-3cm \ell\ =\ \left[ f(t+0) - f(t-0)  \right] \cdot \left[\, \alpha + \chi\left(x\right) \right]\, +\, f(t-0)
$$
$$
\hskip0.2cm  -\ \left[ f(t+0) - f(t-0)  \right]\cdot \chi(x) \cdot x, \hskip1cm \mbox{for every} \hskip1cm x \in (0,1).
$$
Since $f(t+0) - f(t-0) \neq 0$, we can easily obtain:
$$
\chi(x) \cdot (1-x)\ =\ \frac{\ell - f(t-0)}{f(t+0) - f(t-0)}\ -\ \alpha, \hskip0.5cm \mbox{for every} \hskip0.5cm x \in (0,1),
$$
i.e.,
$$
\chi(x)\ =\ \left[\frac{\ell - f(t-0)}{f(t+0) - f(t-0)}\ -\ \alpha\right]\, \cdot \frac{1}{1-x}, \hskip0.5cm \mbox{for every} \hskip0.3cm x \in (0,1),
$$
and this represents a contradiction. Indeed, if 
$$
\frac{\ell - f(t-0)}{f(t+0) - f(t-0)}\ -\ \alpha\ =:\ C \neq 0,
$$
it results that $\chi$ is unbounded on $(0,1)$, hence it fails to satisfy condition $(\chi 1)$ on the kernel functions. While, if $C=0$ it turns out that $\chi(x)=0$ for every $x \in (0,1)$, that is again a contradiction; thus the theorem is proved.
\end{proof}
Finally, the following general theorem can be deduced.
\begin{theorem} \label{th3}
Let $f: \R \to \R$ be a bounded signal with a (non removable) jump discontinuity at $t \in \R$, and let $\alpha \in \R$. Suppose in addition that the kernel function $\chi$ satisfies the following condition:
\be \label{strong2}
\chi(u) = 0,\ \hskip0.8cm for\ every\ \hskip0.8cm x \in [0,1).
\ee
Then, the following assertions are equivalent:
\begin{itemize}
\item[(i)] $\displaystyle \lim_{w \to +\infty} (S_wf)(t)\ =\ \alpha\, f(t+0)\, +\, (1-\alpha)\, f(t-0)$;

\item[(ii)] $ \Psi^-_{\chi}(x)\ =\ \alpha$, \hskip0.5cm for every  \hskip0.5cm $x \in [0, 1)$; 

\item[(iii)] $ \Psi^+_{\chi}(x)\ =\ 1-\alpha$ \hskip0.5cm for every  \hskip0.5cm $x \in [0, 1)$.
\end{itemize}
If we assume in addition that $\chi$ is continuous on $\R$, assertions (i), (ii), and (iii) are also equivalent to the following:
\begin{itemize}
\item[(iv)] 
$\displaystyle \int_{-\infty}^0\chi(u)\, e^{-i u 2 \pi k}\, du\ =\ 
\left\{
\begin{array}{l}
\alpha, \hskip0.5cm k=0, \\
0, \hskip0.5cm k \in \Z\setminus \left\{  0 \right\},\end{array}\right.
$

\item[(v)] 
$\displaystyle \int^{+\infty}_0\chi(u)\, e^{-i u 2 \pi k}\, du\ =\ \left\{
\begin{array}{l}
1-\alpha, \hskip0.5cm k=0, \\
0, \hskip1.2cm k \in \Z\setminus \left\{  0 \right\},\end{array}\right.
$
\end{itemize}
\end{theorem}
\begin{proof}
The equivalence among (i), (ii) and (iii) can be established easily as a consequence of Theorem \ref{th2} and Theorem \ref{th-bis}.

\noindent Let now $\chi$ be continuous on $\R$. We can prove $(ii) \Leftrightarrow (iv)$. Setting:
$$
\chi_0(x):= \left\{
\begin{array}{l}
\chi(x), \hskip0.8cm \mbox{for}\ x<0,\\
0, \hskip1.35cm \mbox{for}\ x \mau 0,
\end{array}
\right.
$$
then it results $\Psi^-_{\chi}(x) = \sum_{k \in \Z} \chi_0(x-k)$ is a 1-periodic, continuous function on $[0,1)$. Now, by the Poisson's summation formula (see e.g., \cite{BUNE,BAMUVI}) its Fourier expansion is given by:
$$
\Psi^-_{\chi}(x)\ =\ \sum_{k \in \Z} \widehat{\chi_0}\, (2k\pi)\, e^{i 2 k \pi x}\ =\ \sum_{k \in \Z}  \left[  \int_{-\infty}^0 \chi(u)\, e^{-i 2 k \pi u} \, du    \right]\, e^{i 2 k \pi x}.
$$
Therefore $\Psi^-_{\chi}(x) = \alpha$ for every $x \in [0,1)$ if and only if its Fourier series reduces to the term $k=0$, and this term is equal to $\alpha$. The above fact implies the equivalence between (ii) and (iv).

\noindent Finally, the equivalence between (iv) and (v) follows immediately from Theorem \ref{th_eq}. This completes the proof.
\end{proof}

  Clearly, also when we refer to (i) of Theorem \ref{th-bis} and Theorem \ref{th3} we speak of convergence of the sampling Kantorovich series at jump discontinuities.

  Finally, we can consider the case in which the discontinuity at the point $t$ is removable, i.e., if $f(t+0)= f(t-0) = \ell$. In the latter case, it is easy to prove the following.
\begin{theorem} \label{th2.7}
Let $f: \R \to \R$ be a bounded signal with a removable discontinuity at $t \in \R$, i.e., $f(t+0)= f(t-0) = \ell$. Then,
$$
\lim_{w \to +\infty} (S_w f)(t)\ =\ \ell.
$$
\end{theorem}
\begin{proof}
The proof follows immediately noting that, by Lemma \ref{lemma1},
$$
(S_w f)(t)\ =\ (S_w g_t)(t)\ +\ \ell,
$$
for every $w>0$, and applying Theorem \ref{th1}.
\end{proof}
Note that, Theorem \ref{th2.7} holds without any additional assumptions on the kernel functions.


\section{Linear prediction and order of approximation} \label{sec3}

 The problem of the order of approximation for the sampling Kantorovich series, has been largely studied in various papers, see e.g. \cite{BAMA,COVI3,COVI4,COVI5}, both for functions of one and several variables. The approach used in the papers \cite{COVI3,COVI4,COVI5} involves functions belonging to suitable Lipschitz classes (see e.g., \cite{BAMUVI}) in the space of continuous functions and in Orlicz spaces. While, the estimates established in \cite{BAMA} were given for continuous functions only, and by employing the modulus of continuity of the function being approximated. We recall that, the modulus of continuity of a given uniformly continuous function $f: \R \to \R$ is defined by:
\be
\omega(f,\, \delta)\ :=\  \sup\left\{ |f(x)-f(y)|:\ x,y \in \R,\ |x-y| \miu \delta \right\},
\ee
$\delta>0$. It is well-known that, for any positive constant $\lambda>0$, the modulus of continuity satisfies the following useful property (see e.g., \cite{BUNE}):
\be \label{property_modulus}
\omega(f, \lambda\, \delta)\ \miu\ (\lambda + 1) \cdot \omega(f, \delta).
\ee
Now, we recall the following result.
\begin{theorem}[\cite{BAMA}] \label{th_bama}
Let $\chi$ be a kernel satisfying condition $(\chi 3)$ with $\beta \mau 1$. Then, for any uniformly continuous and bounded function $f: \R \to \R$, there exists a positive constant $C>0$ such that:
$$
\| S_w f - f\|_{\infty}\ \miu C\, \omega(f, w^{-1}),  \hskip0.8cm w>0.
$$
\end{theorem}
We stress that, the estimate provided in Theorem \ref{th_bama} holds only when $\beta$ of condition $(\chi 3)$ is not less than one. However, there exist examples of kernels for which the discrete absolute moments of order $\beta \mau 1$ are not finite, but at the same time, condition $(\chi 3)$ is satisfied for some values $0< \beta <1$. In the latter case, Theorem \ref{th_bama} cannot be applied. For this reason, we prove the following.
\begin{theorem} \label{th_beta}
Let $\chi$ be a kernel satisfying condition $(\chi 3)$ with $0< \beta < 1$. Then, for any uniformly continuous and bounded function $f: \R \to \R$, we have:
$$
| (S_wf)(x) - f(x) |\ \miu\ \omega \left( f,\, w^{-\beta} \right)\cdot \left[m_{\beta}(\chi) + 2 m_{0}(\chi) \right]\ +\ 2^{\beta + 1}\, \| f\|_{\infty}\, m_{\beta}(\chi)\, w^{-\beta} ,
$$
for every $x \in \R$, and $w>0$ sufficiently large.
\end{theorem}
\begin{proof}
Let $x \in \R$ be fixed. Now we can write what follows:
$$
\hskip-3.3cm |(S_wf)(x) - f(x)|\ =\ \left| (S_wf)(x) - f(x) \sum_{k \in \Z}\chi(wx-k) \right|
$$
$$
\hskip2.4cm \miu\ \sum_{k \in \Z} \left[ w \int^{(k+1)/w}_{k/w} |f(u)-f(x)|\, du \right] \cdot \left| \chi(wx-k) \right|,
$$
for $w>0$. Now, we split the above series as follows:
$$
\hskip-4cm \sum_{|wx-k| \miu w/2} \left[ w \int^{(k+1)/w}_{k/w} |f(u)-f(x)|\, du \right] \cdot \left| \chi(wx-k) \right|
$$
$$
+\ \sum_{|wx-k| > w/2}\left[ w \int^{(k+1)/w}_{k/w} |f(u)-f(x)|\, du \right] \cdot \left| \chi(wx-k) \right|\ =:\ I_1\, +\, I_2.
$$
Before estimating $I_1$, we first observe that, for every $u \in [k/w, (k+1)/w]$, and if $|wx-k| \miu w/2$ we have:
$$
|u-x|\ \miu\ |u-(k/w)|+|(k/w)-x|\ \miu\ (1/w) + (1/2)\ \miu\ 1,
$$
for every $w>0$ sufficiently large, and moreover, since $0 < \beta < 1$, it is also easy to see that:
$$
\omega(f,\, |u-x|)\ \miu\ \omega(f,\, |u-x|^{\beta}).
$$
Hence, by using property (\ref{property_modulus}) we can obtain:
$$
I_1\ \miu\ \sum_{|wx-k| \miu w/2} \left[ w \int^{(k+1)/w}_{k/w} \omega(f,\, |u-x|^{\beta})\, du \right] \cdot \left| \chi(wx-k) \right|
$$
$$
\miu\ \sum_{|wx-k| \miu w/2} \left[ w \int^{(k+1)/w}_{k/w} \left( w^{\beta}|u-x|^{\beta} + 1 \right)\, \omega(f,\, w^{-\beta})\, du \right] \cdot \left| \chi(wx-k) \right|
$$
$$
\miu\ \omega(f,\, w^{-\beta}) \left[ \sum_{|wx-k| \miu w/2}\left( w \int^{(k+1)/w}_{k/w}w^{\beta}|u-x|^{\beta}\, du \right)\left| \chi(wx-k) \right| \right.
$$
$$
+\ \left. \sum_{|wx-k| \miu w/2}\left| \chi(wx-k) \right| \right]\ =:\ \omega(f,\, w^{-\beta})\, \left[ I_{1,1}\, +\ I_{1,2}\right|.
$$
Concerning $I_{1,1}$, exploiting the sub-additivity of $|\cdot|^{\beta}$, with $0<\beta<1$, we have:
$$
I_{1,1}\ \miu\ \sum_{|wx-k| \miu w/2}\left( w^{\beta}\, \max_{u \in [k/w,\, (k+1)/w]}|u-x|^{\beta} \right)\left| \chi(wx-k) \right|
$$
$$
\miu\ \sum_{|wx-k| \miu w/2}w^{\beta}\, \max\left\{  |(k/w)-x|^{\beta};\, |(k+1)/w - x|^{\beta} \right\}\,  \left| \chi(wx-k) \right|
$$
$$
\miu\ \sum_{|wx-k| \miu w/2}w^{\beta}\, \max\left\{  |(k/w)-x|^{\beta};\, |(k/w) - x|^{\beta} + (1/w)^{\beta} \right\}\,  \left| \chi(wx-k) \right|
$$
$$
=\ \sum_{|wx-k| \miu w/2}w^{\beta}\, \left[  |(k/w) - x|^{\beta} + w^{-\beta} \right]\,  \left| \chi(wx-k) \right|
$$
$$
\miu\ \sum_{|wx-k| \miu w/2}|wx-k|^{\beta}\left| \chi(wx-k) \right|\ +\ \sum_{|wx-k| \miu w/2}\left| \chi(wx-k) \right|
$$
$$
\miu\ m_{\beta}(\chi)\, +\, m_{0}(\chi)\ <\ +\infty,
$$
in view of condition $(\chi 3)$ and what observed in Remark \ref{remark2.3}.
Moreover, follows easily that also $I_{1,2} \miu m_{0}(\chi)$. Finally, concerning $I_2$ we have:
$$
I_2\ \miu\ 2 \| f\|_{\infty} \sum_{|wx-k| > w/2}\left| \chi(wx-k) \right| \miu\ 2 \| f\|_{\infty} \sum_{|wx-k| > w/2}\frac{|wx-k|^{\beta}}{|wx-k|^{\beta}}\left| \chi(wx-k) \right|
$$
$$
\miu\ 2^{\beta + 1}\, \frac{\| f\|_{\infty}}{w^{\beta}}\sum_{|wx-k| > w/2} |wx-k|^{\beta}\, \left| \chi(wx-k) \right|\ \miu\ 2^{\beta + 1}\, \| f\|_{\infty} w^{-\beta} m_{\beta}(\chi)\ <\ +\infty.
$$
This completes the proof.
\end{proof}
After estimating the order of approximation, we investigate the problem of the linear prediction of signals, by sample values taken only from the past. In fact, the sampling Kantorovich series, in order to approximate a signal at a fixed time $t$, involve sample values taken both from the past and the future with respect to $t$. Clearly, in practice this is not possible since a signal is known only from the past with respect to the time $t$. The above problem can be solved as follows.
\begin{theorem} \label{th_prediction}
Let $\chi$ be a kernel with compact support, such that $supp\, \chi \subset (0,+\infty)$. Then, for every signal $f: \R \to \R$ for which the operators $S_w f$, $w>0$, are well-defined, and for every fixed $t \in \R$, we have:
$$
(S_w f)(t)\ =\ \sum_{k/w\, <\, t} \chi(wt-k) \left[ w \int_{k/w}^{(k+1)/w} f(u)\, du \right],
$$
for every $w>0$, such that $w t \in \Z$.

\noindent In particular, if $supp\, \chi \subset [1,+\infty)$ we have:
$$
(S_w f)(t)\ =\ \sum_{k/w\, <\, t} \chi(wt-k) \left[ w \int_{k/w}^{(k+1)/w} f(u)\, du \right],
$$
for every $w>0$.
\end{theorem}
\begin{proof}
Since $supp\, \chi \subset (0,+\infty)$, we have $\chi(wt-k) = 0$ for every $k \in \Z$ such that $wt-k \miu 0$, i.e., $k/w\, \mau\, t$. Moreover, if $w>0$ is such that $wt \in \Z$, it turns out that the last term of the series $(S_wf)(t)$ is: 
$$
\chi(1) \cdot \left[ w \int_{t-(1/w)}^{t} f(u)\, du \right];
$$
therefore it is clear that our operators exploit the values of the signal $f$ in the past with respect to the fixed time $t$.

\noindent Moreover, if $supp\, \chi \subset [1,+\infty)$ the proof follows as before, but we must observe that, if $wt \notin \Z$, we have $\lef wt \rig < wt < \lef wt \rig+1$, then:
$$
\sum_{k/w\, <\, t}\!\!\! \chi(wt-k) \left[ w \int_{k/w}^{(k+1)/w} f(u)\, du \right] = \!\!\!  \sum_{k\, \miu\, \lef wt \rig -1}\!\!\!  \chi(wt-k) \left[ w \int_{k/w}^{(k+1)/w} f(u)\, du \right]
$$
$$
+ \chi(wt- \lef wt \rig)\! \left[ w \int_{\lef wt \rig/w}^{(\lef wt \rig+1)/w}\!\!\! f(u)\, du \right] \!\! = \!\!\!  \sum_{k\, \miu\, \lef wt \rig -1}\!\!\!\!\!  \chi(wt-k) \left[ w \int_{k/w}^{(k+1)/w} f(u)\, du \right], 
$$
since $wt- \lef wt \rig<1$; then again the mean values are computed before the time $t$.
\end{proof}
Obviously, it is easy to see that, if the kernel $\chi$ has compact support, the above sampling Kantorovich series reduce to finite sums.


\section{Examples of kernels and particular cases} \label{sec4}

In this section, we present some concrete examples of kernels. First of all, we mention some well-known, band-limited kernels:
$$
\hskip-1.6cm \chi_1(x)\ :=\ \frac{1}{2}\, \left(\frac{\sin(\pi x /2)}{\pi x / 2}\right)^2, \hskip0.9cm x \in \R, \hskip0.9cm \mbox{(Fej\'er's kernel),}
$$
$$
\hskip0.33cm \chi_2(x)\ :=\ \frac{3}{2\pi}\, \frac{\sin(x /2)\, \sin(3 x /2)}{3 x^2 / 4}, \hskip0.3cm x \in \R, \hskip0.3cm \mbox{(de la Vall\'ee Poussin's kernel),}
$$
$$
\hskip-4.7cm \chi_3(x)\ :=\ \frac{\sin(\pi x /2)\, \sin(\pi x)}{\pi^2 x^2/2}, \hskip1.2cm x \in \R,
$$
see e.g. \cite{BUNE,BURIST1,BABUSTVI,COVI1,COVI2}.
For the above examples we have that condition $(\chi 1)$ is easily fulfilled, condition $(\chi 2)$ is satisfied in view of Theorem \ref{th_eq}, and finally condition $(\chi 3)$ is satisfied for every $\beta < 1$, see e.g. \cite{BUNE,BURIST1,BAMUVI,BABUSTVI}.

  For the sampling Kantorovich operators associated to these kernels, we can clearly speak of convergence in the standard sense (that one established in Theorem \ref{th1}), and we can speak of convergence at jump discontinuities only in the case considered in Theorem \ref{th2}. While, Theorem \ref{th-bis} and Theorem \ref{th3} cannot be applied, since in general $\chi_j(x) \neq 0$ for every $x \in (0,1)$, $j=1,2,3$.

   Further, concerning the order of approximation that can be achieved by the operators $S_w$, $w>0$, when the above kernels are employed, it is possible to observe that, since condition $(\chi 3)$ is satisfied only for every $\beta < 1$ (then $m_{\beta}(\chi_i)=+\infty$, for $\beta \mau 1$, $i=1,2,3$), Theorem \ref{th_bama} cannot be applied, while Theorem \ref{th_beta} holds. 

 As a general fact, we can observe that, it is impossible to construct non-trivial band-limited kernels, such that $\chi(x)=0$ for every $x \in (0,1)$, since they must be identically zero, due to the Paley-Wiener theorem, which establish the holomorphy of the function $\chi$ and the identity principle of the analytic functions.

  Now, since it is well-known that continuous functions cannot  be simultaneously duration and band limited, in order to obtain examples of kernels for which the corresponding sampling Kantorovich operators converge at jump discontinuity, the above remark suggests to consider only duration limited kernels, i.e., kernels with compact support.
\begin{theorem} \label{th_duration}
Let $\chi_a,\, \chi_b: \R \to \R$ be two continuous kernels such that, $supp\, \chi_a \subseteq [-a,a]$ and $supp\, \chi_b \subseteq [-b,b]$, with $a$ and $b$ positive. Moreover, let $\alpha \in \R$ be fixed. Then, setting:
$$
\chi(x)\ :=\ (1-\alpha)\, \chi_a(x-a-1)\ +\ \alpha\, \chi_b(x+b), \hskip1cm x \in \R,
$$
it turns out that $\chi$ is a kernel satisfying conditions $(\chi 1)$, $(\chi 2)$, and $(\chi 3)$. Furthermore, the corresponding sampling Kantorovich series $S_w f$, $w>0$, based upon $\chi$ satisfy (i) of Theorem \ref{th3} with the parameter $\alpha$, at any jump discontinuity $t \in \R$ of the given bounded signal $f: \R \to \R$.
\end{theorem}
\begin{proof}
First of all, it is easy to see that $\chi$ satisfy conditions $(\chi 1)$ and $(\chi 3)$ (see e.g. \cite{BABUSTVI}). Moreover, since $\chi$ is defined by a suitable finite linear combination of two continuous kernels with compact support, it turns out that $\chi$ is itself continuous and with compact support. Further, its Fourier transform can be computed as follows:
$$
\widehat{\chi}(v)\ =\ (1-\alpha)\, e^{-i v\, (a+1)}\, \widehat{\chi_a}(v)\ +\ \alpha \,e^{i v\, b}\, \widehat{\chi_b}(v), \hskip0.5cm v \in \R.
$$
Now, by Theorem \ref{th_eq}, and since $\chi_a$ and $\chi_b$ satisfy condition $(\chi 2)$, we have $\widehat{\chi}(0)=1$, and $\widehat{\chi}(2\pi k) = 0$, for $k \in \Z \setminus\left\{ 0 \right\}$, then also $\chi$ satisfies $(\chi 2)$ and so it turns out to be a kernel function. In addition, it is easy to see that, for any $x \in [0,1)$, we have $\chi(x)=0$, i.e., $\chi$ satisfies condition $(\ref{strong})$. Finally, we can observe that:
$$
\hskip-1.2cm \int^{+\infty}_0 \chi(u)\, e^{-iu2\pi v} du\ =\ (1-\alpha) \int^{+\infty}_0 \chi_a(u-a-1)\, e^{-iu2\pi v}\, du
$$
$$
\hskip0.8cm =\ (1-\alpha)\, e^{-i v\, 2\pi\, (a+1)}\, \widehat{\chi_a}(2\pi v),
$$
and it results $\widehat{\chi_a}(2 \pi k) e^{-i 2 \pi k\, (a+1)}=1$, if $k=0$, and $\widehat{\chi_a}(2 \pi k) e^{-i 2 \pi k\, (a+1)}=0$, if $k \in \Z \setminus\left\{ 0 \right\}$, then we obtain that condition (v) of Theorem \ref{th3} is fulfilled. This completes the proof.
\end{proof}

   In order to construct examples of kernels as in Theorem \ref{th_duration}, we first recall the definition of the well-known central B-spline order $n \in\ \N^+$, given by:
$$
    M_n(x)\ :=\ \frac{1}{(n-1)!} \sum^n_{i=0}(-1)^i \binom{n}{i} 
       \left(\frac{n}{2} + x - i \right)^{n-1}_+,     \hskip0.5cm   x \in \R,
$$
where $(x)_+ := \max\left\{x,0 \right\}$ denotes ``the positive part'' of $x \in \R$. Clearly, central B-splines of order $n$ does not satisfy the condition $\chi(x)=0$, for every $x \in[0,1)$. 

  It is well-know that, any $M_n(x)$ has support $[-n/2, n/2]$; hence it is bounded and belonging to $L^1(\R)$. Thus condition $(\chi 1)$ holds, $(\chi 3)$ is easily satisfied for every positive values of $\beta$, and noting that:
$$
\widehat{M}_n(v)\ =\ \left( \frac{\sin v/2}{v/2} \right)^n, \hskip0.8cm v \in \R,
$$
we obtain from Theorem \ref{th_eq} that also condition $(\chi 2)$ is fulfilled. Now, for every central B-spline of order $n \in \N$, we define for each $\alpha \in \R$, the following kernels:
$$
\chi_n(x)\ :=\ (1-\alpha)\, M_n(x-n-1)\, +\, \alpha\, M_n(x+n), \hskip0.8cm x \in \R.
$$
By means of the kernels $\chi_n$, the corresponding sampling Kantorovich series $S_w f$, $w>0$, satisfy (i) of Theorem \ref{th3} with the parameter $\alpha$, at any jump discontinuity $t \in \R$ of a given bounded signal $f: \R \to \R$.

  Until now, we showed only examples of continuous kernels; in what follows we introduce a discontinuous kernel. Let
$$
C_2(x)\ :=\ \frac{1}{2}\left[ M_2(x+2)+M_2(x-2) \right]\ =\ \left\{ 
\begin{array}{l}
(|x|-1)/2, \hskip0.4cm 1 \miu x < 2, \\
(3-|x|)/2, \hskip0.4cm 2 \miu x < 3, \\
0, \hskip2cm \mbox{elsewhere},
\end{array}
\right.
$$
a continuous kernel constructed by a procedure similar to that one described in the proof of Theorem \ref{th_duration}. Clearly, the sampling Kantorovich operators based upon $C_2$ satisfy (i) of Theorem \ref{th3} with $\alpha = 1/2$. Now, we introduce the following step function:
$$
S(x)\ :=\ \left\{ 
\begin{array}{l}
1/2, \hskip1.2cm |x|=1, \\
-1/2, \hskip0.9cm |x|=2 \\
0, \hskip1.61cm \mbox{elsewhere}.
\end{array}
\right.
$$
It is easy to check that $\displaystyle \sum_{k \in \Z} S(x-k) =0$ and $\Psi^-_S(x)=0$, for every $x \in [0,1)$; now defining:
$$
D_2(x)\ :=\ C_2(x)\, +\, S(x), \hskip1cm x \in \R,
$$ 
we obtain that the discontinuous function $D_2(x)$ is again a kernel satisfying $(\chi 1)$, $(\chi 2)$, and $(\chi 3)$, $D_2(x)=0$ for every $x \in [0,1)$, and it satisfies condition (ii) of Theorem \ref{th3} (in the not necessarily continuous case). Hence, the corresponding sampling Kantorovich series based upon $D_2$ converge at jump discontinuity with $\alpha=1/2$. 

 Note that, while the previous examples of kernels with compact support have $m_{\beta}(\chi)<+\infty$, for every $\beta \mau 0$, the kernels $\chi_1$, $\chi_2$ and $\chi_3$ have $m_{\beta}(\chi_i)=+\infty$, for $\beta \mau 1$, $i=1,2,3$. In the latter case, Theorem \ref{th_bama} cannot be applied and Theorem \ref{th_beta} becomes useful in order to estimate the aliasing errors for the sampling Kantorovich operators.

   In the following theorem, we show that, in general, it is possible to give a condition on the kernels which ensures that $(\chi 3)$ holds for $0\miu \beta <\nu$, for some $\nu <1$, and $m_\beta(\chi)=+\infty$, for $\nu < \beta \miu 1$. 
\begin{theorem} \label{th4.3}
Let $\chi: \R \to \R$ be a function such that the following condition holds for suitable constants $0< C_1 \miu C_2$:
$$
\frac{C_1}{|u|^{\gamma}}\ \miu\ \chi(u)\ \miu\ \frac{C_2}{|u|^{\gamma}},\ \hskip0.5cm for\ every\ \hskip0.5cm |u| > M,
$$
for some $1< \gamma \miu 2$, and $M>0$. Then:
$$
m_{\beta}(\chi)\ =\ \left\{
\begin{array}{l}
+\infty, \hskip1.4cm \gamma-1\ \miu\ \beta\ \miu 1, \\
<  +\infty,  \hskip1cm  0\ \miu\ \beta\ <\ \gamma - 1.
\end{array}
\right.
$$
\end{theorem}
\begin{proof}
Let $\gamma-1 \miu \beta \miu 1$, and $u \in \R$ be fixed. We can write:
$$
m_\beta(\chi)\ \mau\ \sum_{k \in \Z}|\chi(u-k)|\, |u-k|^{\beta}\ \mau\ C_1 \sum_{|u-k|>M} |u-k|^{\beta-\gamma}\ =\ +\infty,
$$
since $\gamma-\beta \miu 1$. While, for $0 \miu\ \beta< \gamma - 1$, and every $u \in \R$ we have:
$$
\sum_{k \in \Z}|\chi(u-k)|\, |u-k|^{\beta}\ \miu\ \left\{ \sum_{|u-k| \miu M} + \sum_{|u-k|>M}  \right\}|\chi(u-k)|\, |u-k|^{\beta}
$$
$$
\miu\  M^\beta m_0(\chi)\ +\ C_2\, \sum_{|u-k|>M} |u-k|^{\beta-\gamma}\ < \ +\infty,
$$
where $m_0(\chi) <+\infty$ since $\chi(u)={\cal O}(|u|^{-\gamma})$, as $|u| \to +\infty$ with $\gamma>1$, and since the series $\sum_{|u-k|>M} |u-k|^{\beta-\gamma}$, with $\gamma - \beta >1$, is uniformly convergent for every $u \in \R$. Then the proof follows.
\end{proof}

  Examples of kernels as in Theorem \ref{th4.3}, which satisfy $(\chi 1)$, $(\chi 2)$, and $(\chi 3)$ with $0 <\beta<1$, can be constructed by using certain finite linear combination of sigmoidal functions, see e.g. \cite{COSP1,COSP2,CO1,COSP4,COSP3,CO2,COVI7}.

  We recall that, a function $\sigma: \R \to \R$ is called a sigmoidal function if it satisfies the following:
$$
\lim_{x\to-\infty}\sigma(x)=0, \hskip0.8cm \mbox{and} \hskip0.8cm \lim_{x\to+\infty}\sigma(x)=1,
$$
see e.g., \cite{CY,CO1,COSP3,CO2}.
In particular, we consider the following sigmoidal functions, first introduced in \cite{COVI7}:
\be
\sigma_{\gamma}(x)\ :=\ 
\left\{
\begin{array}{l}
\disp \frac{1}{|x|^{\gamma}+2}, \hskip2.7cm x<- 2^{1/\gamma},\\
\\
\disp 2^{-(1/\gamma)-2}x+(1/2), \hskip0.6cm -  2^{1/\gamma} \miu x \miu 2^{1/\gamma},\\ 
\\
\disp \frac{x^{\gamma}+1}{x^{\gamma}+2}, \hskip2.9cm x > 2^{1/\gamma},
\end{array}
\right.
\ee
with $1 < \gamma \miu 2$. Now, we can define the positive functions:
$$
\phi_\gamma(u)\ :=\ \frac{1}{2}\, [\sigma_{\gamma}(u+1)-\sigma_{\gamma}(u-1)], \hskip0.8cm u \in \R,
$$
which satisfies the assumption of Theorem \ref{th4.3} for $1 < \gamma \miu 2$, then it turns out that:
$$
m_{\beta}(\phi_\gamma)\ =\ \left\{
\begin{array}{l}
+\infty, \hskip1.4cm \gamma-1\ \miu\ \beta\ \miu 1, \\
<  +\infty,  \hskip1cm  0\ \miu\ \beta\ <\ \gamma - 1,
\end{array}
\right.
$$
i.e., $\phi_\gamma$ satisfies $(\chi 3)$ for $0 \miu \beta < \gamma - 1<1$. Moreover, in \cite{COVI7} (or in \cite{COSP1,COSP2,COSP4}) it is showed that $\phi_\gamma$ satisfies also conditions $(\chi 1)$ and $(\chi 2)$.


\section{Final remarks and conclusions} \label{sec5}

Some necessary and sufficient conditions for the convergence at jump and removable discontinuities for the families of sampling Kantorovich operators have been established. A crucial role is played by suitable kernels a class of which
can be easily obtained by a certain finite linear combination of duration-limited kernels. 

   We note that, in the present setting, the representation formula of Lemma \ref{lemma1} is more delicate to achieve in comparison with the analogous one in \cite{BURIST1}, due to the presence of the average instead of the sample values $f(k/w)$, $k \in \Z$, $w>0$.

   Finally, since in some previous papers (see e.g. \cite{COVI1,COVI2,CLCOMIVI1,CLCOMIVI2}) it has been proved that the sampling Kantorovich operators are suitable for image reconstruction and enhancement, and for this aim it is important to study the behavior of the image in the point of discontinuity, the present study can be useful for its extension to the multivariate frame.

\vskip0.2cm

\section*{Acknowledgment}

The authors are members of the Gruppo  
Nazionale per l'Analisi Matematica, la Probabilità e le loro  
Applicazioni (GNAMPA) of the Istituto Nazionale di Alta Matematica (INdAM).

\noindent The first and the second author of the paper have been partially supported within the GNAMPA-INdAM Project ``Metodi di approssimazione e applicazioni al Signal e Image Processing''; project number: U2015/000396 12/03/2015.

\vskip0.1cm
%
%

\end{document}